\documentclass[11pt,twoside]{amsart}
\usepackage{bbm}
\usepackage{enumerate}
\usepackage{subcaption}
\usepackage{xspace}

\usepackage{graphicx,color}
\usepackage[all]{xy}
\usepackage{verbatim}
\usepackage[dvipsnames]{xcolor}

\usepackage{eso-pic}

\theoremstyle{plain}
\newtheorem{The}{Theorem}
\newtheorem*{The*}{Theorem}
\newtheorem{Pro}{Proposition}
\newtheorem{Lem}{Lemma}
\newtheorem{Cor}{Corollary}
\newtheorem*{Cor*}{Corollary}

\theoremstyle{definition}

\newtheorem{Rem}{Remark}

\newtheorem*{Rem*}{Remark}

\numberwithin{equation}{section}

\DeclareMathOperator{\GL}{GL}
\DeclareMathOperator{\SL}{SL}

\DeclareMathOperator{\SO}{SO}

\DeclareMathOperator{\Span}{Span}

\newcommand{\R}{\mathbb{R}}

\newcommand{\C}{\mathbb{C}}
\newcommand{\N}{\mathbb{N}}
\newcommand{\Z}{\mathbb{Z}}
\renewcommand{\P}{\mathbb{P}}
\renewcommand{\H}{\mathbb{H}}  

\newcommand{\CP}{\mathbb{CP}}

\newcommand{\jj}{\mathbbm j}

\begin{document}

\title{Isothermic constrained Willmore tori in $3$-space}

\author{Lynn Heller}

\address{ Institut f\"ur Differentialgeometrie\\  Leibniz Universit{\"a}t Hannover\\ Welfengarten 1, 30167 Hannover}

 \email{lynn.heller@math.uni-hannover.de}

 \author{Sebastian Heller}
\address{Department of Mathematics\\
University of Hamburg\\
20146 Hamburg, Germany
 }
 \email{seb.heller@gmail.com}

\author{Cheikh Birahim Ndiaye}
\address{Department of Mathematics of Howard University\\
204 Academic Support Building B
Washington, DC 20059k \\ USA
 }
 \email{cheikh.ndiaye@howard.edu}


 \date{\today}

\noindent

\begin{abstract} 
We show that the  homogeneous and the $2$-lobe Delaunay tori in the $3$-sphere provide the only isothermic constrained Willmore tori in $3$-space with Willmore energy below $8\pi$. In particular, every constrained Willmore torus with Willmore energy below $8\pi$ and non-rectangular conformal class is non-degenerated.
 \end{abstract}

\maketitle



\section{Introduction}
The Willmore functional of an 
immersions $f\colon M\to S^3$ from a oriented surface $M$ into the $3$-sphere is given by
\[
\mathcal{W}(f)=\int_M (H^2+1) dA
\] 
where $H$ is the  mean curvature and $dA$ is the induced area form of $f$. Geometrically speaking $\mathcal W$ measures the roundness of a surface, physically  the degree of bending, and in biology $\mathcal W$ appears as a special instance of the Helfrich energy for cell membranes.  The Willmore functional is invariant 
under Moebius transformations (conformal transformations of the $3$-sphere with its standard conformal structure).
Critical points of the Willmore functional are Willmore surfaces. Examples are given by minimal surfaces in the
Riemannian subgeometries of constant curvature of the conformal 3-sphere.

If $M$ is equipped with a Riemann surface structure, it is natural to consider only conformal immersions $f\colon M\to S^3$,
i.e., the complex structure is given by rotating tangent vectors by $\tfrac{\pi}{2}$ in the 3-space. Critical points of the
Willmore functional restricted to a given conformal class are called constrained Willmore surfaces.
The conformal constraint augments the Euler-Lagrange equation by a holomorphic quadratic differential  $\omega\in H^0(K^2_M)$ paired with the trace-free second fundamental  form $\mathring{A}$ of the immersion
\[
\triangle H+ 2H(H^2+1-K)=\,<\omega,\mathring{A}>,
\]
see \cite{BPP,Schaetzle2}.
The first examples of these constrained Willmore tori are given by those of constant mean curvature (CMC) in a $3$-dimensional space form.

It is well-known (and obvious by the holomorphicity of the Hopf differential) that CMC (constant mean curvature) surfaces admit conformal curvature line parametrizations away from their umbilical points. Surfaces with this property are called {\em isothermic}.
Isothermic surfaces play an important role in conformal surface geometry, see  \cite{BuPP,BuCa}, since the notion is independent of the specific metric in the conformal class of the ambient manifold.
For a compact surface $M$, there is a natural map from the space of immersions into the 3-space to the Teichm\"uller space
\[\pi \colon f\in \text{Imm}(M,S^3)\mapsto [f^*g_{round}]\in \text{Teich}(M),\]
where $g_{round}$ is the round metric on $S^3,$ and $[.]$ denotes the conformal class in the Teichm\"uller space.
The map $\pi$ is a submersion except at isothermic immersions, see \cite{BPP}. Hence, the Lagrange multiplier for isothermic constrained Willmore surfaces  -- the holomorphic quadratic differential -- is no longer uniquely determined by the immersion.  
  
  In this paper, we restrict to compact Riemann surfaces of genus 1. We classify  isothermic constrained Willmore tori
  with Willmore energy below $8\pi.$ Our main theorem is the following one.

\begin{The}\label{classification}
Isothermic constrained Willmore tori in the conformal 3-sphere with Willmore energy below $8 \pi$ are CMC surfaces in the round 3-sphere.
\end{The}

 \begin{figure}
\centering
\includegraphics[width=0.375\textwidth
]
{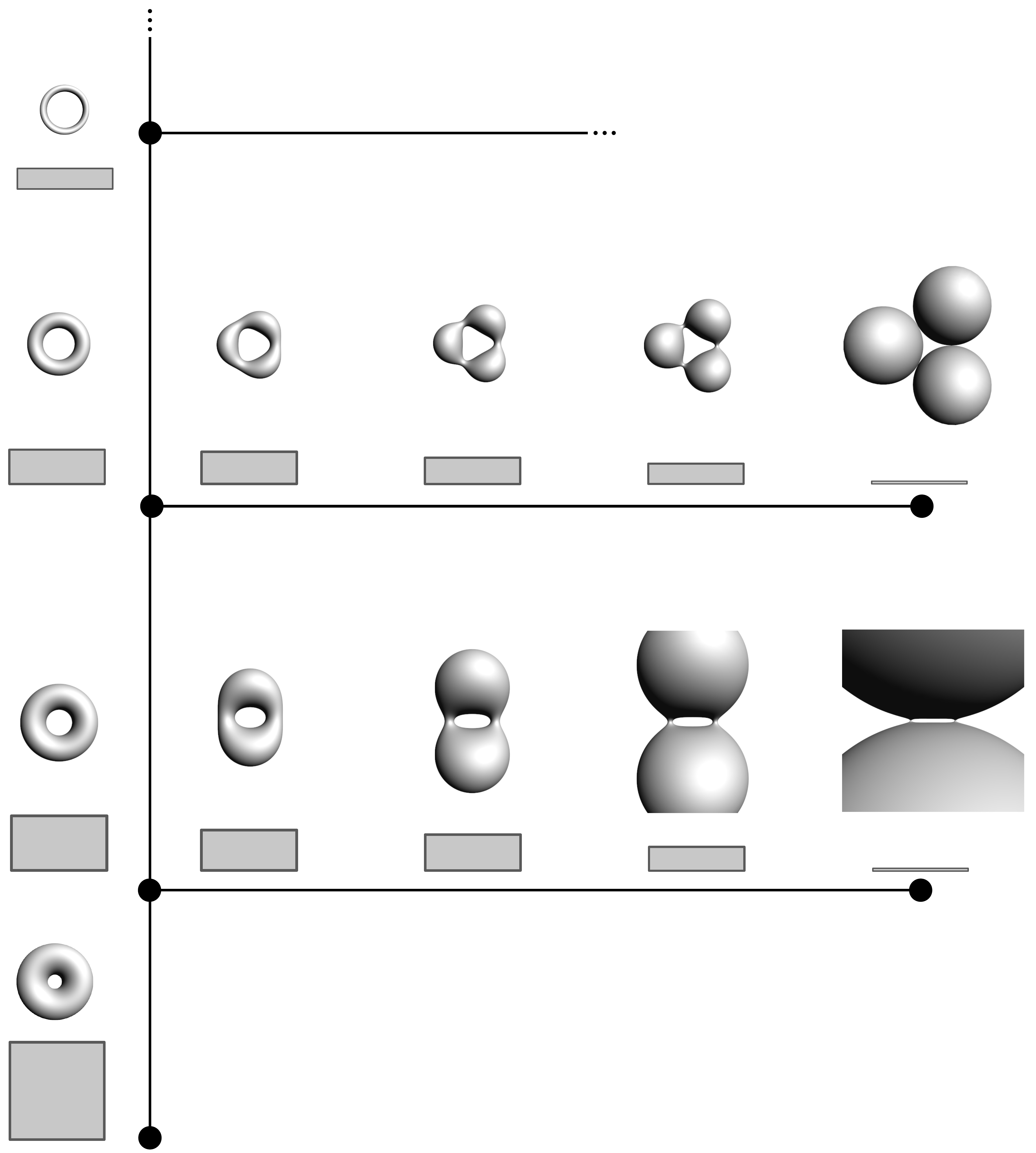}
\caption{
The vertical stalk represents the family of homogenous tori,
starting with the Clifford torus at the bottom.
Along this stalk are bifurcation points at which
the embedded Delaunay tori appear along the horizontal lines.
The rectangles indicate the conformal types. Images by Nicholas Schmitt.
}
\label{fig:torus-tree}
\end{figure}

\subsection*{Strategy of proof}

Richter \cite{Richter} shows that isothermic constrained Willmore tori in the conformal $3$-sphere are locally of constant mean curvature in a $3$-dimensional space form. The solution of the  Lawson and Pinkall-Sterling conjectures by Brendle \cite{Brendle} and Andrews-Li \cite{AndrewsLi} further gives that embedded CMC tori in the $3$-sphere are rotationally symmetric and thus consist  only of the families of $k$-lobed Delaunay tori \cite{KilianSchmidtSchmitt1}. Moreover, the Willmore energy along every embedded family is monotonically increasing in the conformal class $b$. Thus since for $k\geq3$ the $k$-lobes bifurcates from the homogenous tori with Willmore energy above $8\pi$, the $2$-lobed family is the CMC-family with minimal Willmore energy in their respective conformal classes. The aim is to exclude the existence of constrained Willmore surfaces of 
constant mean curvature in $\R^3$ or hyperbolic 3-space $\mathcal H^3$ that can be compactified to a torus in $S^3$ with Willmore energy below $8 \pi.$ By Li and Yau \cite{LiYau} these surfaces must be embedded.\\

The Alexandrov maximum principle \cite{Alexandrov} shows that there are no closed CMC tori with Willmore energy below $8\pi$ in $\R^3$ or  $\mathcal H^3$. The only non-closed CMC surfaces in $\R^3$ that can be compactified to conformal embeddings in $S^3$ are minimal surfaces with planar ends ($H\neq0$ is excluded by local analysis  \cite{KoKuSo}), which have quantized energy $4 \pi k,$ with $k\geq2$ being the number of ends. Thus, those surfaces have Willmore energy $\geq 8\pi$. Similar arguments work for constant mean curvature surfaces in $\mathcal H^3$ with mean curvature $H = 1$ giving  quantized Willmore energy $\mathcal W=4 \pi k,$
where $k\in\N$ denotes the number of ends see \cite{BohPet} and one-punctured CMC 1 torus in $\mathcal H^3$ does not exist by \cite{Pirola}.

To prove Theorem \ref{classification} it is thus sufficient to show that isothermic constrained Willmore tori in $S^3$,
whose intersection with $\mathcal H^3\subset S^3$ is 
of constant mean curvature,  cannot have Willmore energy below $8 \pi$. 
Those surfaces intersect the infinity boundary of  $\mathcal H^3\subset S^3$ -- a round 2-sphere -- with an angle $\alpha$ satisfying $\cos(\alpha) = H. $ In particular, the constant mean curvature must satisfy $|H| <1$ or the surface is entirely contained in $\mathcal H^3,$ and therefore cannot be embedded  by  maximum principle. It hence remains to show that CMC surfaces in $\mathcal H^3$ with mean curvature $|H| <1$ and Willmore energy below $8 \pi$ cannot be embedded, see Theorem \ref{H3}.
 
We will call isothermic constrained Willmore tori into $S^3$ which are CMC in $\mathcal H^3$ with $|H| <1$ on the intersection with the two hyperbolic balls Babich-Bobenko tori in the following.  The first examples have been constructed by Babich and Bobenko \cite{BabBob} in the case of $H=0$.  The main idea of the proof is now to use the quaternionic Pl\"ucker estimate \cite{FLPP}, which links lower bounds of the Willmore energy to the dimension of holomorphic sections of a certain quaternionic holomorphic vector bundle. This dimension is then related to the (necessarily odd) genus $g$ of the spectral curve for Babich-Bobenko tori.

The paper is organized as follows: In Section 2
we study the spectral curve of Babich-Bobenko tori in detail. In Section 3, we use the special structure of the spectral curve
to apply the Pl\"ucker estimate which yields a proof of Theorem \ref{H3}.

\subsection*{Acknowledgements}  The first author is supported by the DFG within the SPP {\em Geometry at Infinity}, and the second author is supported by RTG 1670 {\em Mathematics inspired by string theory and quantum field theory} funded by the  DFG. The second author would also like to thank the International Centre for Theoretical Sciences (ICTS), Bangalore, for hospitality during  the ICTS program on {\em Analytic and Algebraic Geometry}, where parts of
the computations have been performed.

\section{The constrained Willmore spectral curves of Babich-Bobenko tori} \label{isothermic}

We consider two different approaches to the spectral curve theory of Babich-Bobenko tori. The aim of this section is to show that these two approaches towards the spectral curve are in fact equivalent. The lightcone model one is used to show that the spectral curve of a Babich-Bebenko torus  -- the Riemann surface parametrizing the eigenlines of $d^\lambda_q$ --  is hyperelliptic, while the Pl\"ucker estimate uses the multiplier spectral curve, which by \cite{Bohle} corresponds to the spectral curve of $\nabla^\mu$ from the quaternionic approach. Subtleties arise from the non-uniqueness of the Lagrange multipliers.

 \subsection{Quaternionic geometry} The spectral curve theory for conformal immersions $f$ from a $2$-torus $T^2$ into the conformal 4-sphere has been 
 developed in \cite{BFLPP}, where $S^4$ is considered as the quaternionic projective space $\H P^1$. To every conformal 
 immersion $f$ the quaternionic  line bundle \[L= f^* \mathcal T \subset T^2 \times \H^2\] given by the pull-back 
 of the tautological bundle $\mathcal T$ of $\H P^1$ is associated. Another quaternionic  line bundle associated to $f$ is $V/L,$ where $V= T^2 \times \H^2.$ On $V/L$ there exists a natural quaternionic holomorphic structure $D$ (see \cite{BFLPP,Bohle} for a detailed definition and discussion) by demanding the projections of the constant sections $(0,1)$ and $(1,0)$ of $V$ to be holomorphic. The immersion $f$ is then recovered (up to conformal transformations) by 
\[[0,1] = - [1,0] f.\]

The (multiplier) spectral curve $\Sigma$ of a conformally immersed torus $f$ is the normalization of the Riemann surface parametrizing all holomorphic sections of $V/L$ with (complex) monodromy, i.e, every point of the spectral curve corresponds to a holomorphic section with monodromy \cite{BLPP}. Therefore, we can define maps from $\Sigma$ to $\C$ -- so-called monodromy maps $\nu_i$ -- by assigning to every point in $\Sigma$ the monodromy of the underlying holomorphic section along generators $\gamma_i$ of the fundamental group $\pi_1(T^2)$. 

Bohle \cite{Bohle} gives an alternative approach to the spectral curve for constrained Willmore tori. 
For constrained Willmore surfaces $f\colon M\longrightarrow S^3\subset S^4$ Bohle  defined the following $\C_*$-family of flat SL$(4,\C)$-connections
\begin{equation}\label{nabmu}\nabla^\mu=d+(\mu-1)A_\circ^{1,0}+(\mu^{-1}-1)A_\circ^{0,1}.\end{equation}
Here $d$ is the trivial connection on the trivial $\H^2$-bundle considered as a $\C^4$-bundle, $$A_\circ=A+*\eta$$
 where $A$ is the Hopf field of the conformal immersion and $q$ is the Lagrange-multiplier of the constrained Willmore Euler-Lagrange equation (which is not unique for isothermic surfaces). 
He showed that the flatness of an associated $\C^*$-family $\nabla^\mu$ of $\SL(4, \C)$-connections defined on the trivial bundle $V$, considered as a $\C^4$-bundle, is equivalent to $f$ being constrained Willmore. The (holonomy) spectral curve is then given by the Riemann surface parametrizing the eigenlines of the holonomy of $\nabla^\mu.$ Bohle \cite{Bohle} showed that the (holonomy) spectral curve is always of finite genus and that  both approaches to the spectral curve coincide. To be more precise, Bohle showed that $\nabla^\mu$-parallel sections with monodromy are the unique prolongations of the holomorphic sections with monodromy of the quaternionic holomorphic line bundle $(V/L, D)$ to $V.$ The genus $g$ of the associated spectral curve is called the spectral genus of the immersion $f$.

\begin{Rem}
In the case of $f$ mapping into the $3$-sphere $S^3\subset S^4$, the spectral curve $\Sigma$ admits an additional involution $\sigma$, see \cite[Lemma 1]{Hel3}. Another involution $\rho $ on $\sigma$ arises from the quaternionic construction, i.e., by an appropriate multiplication by $\jj.$ If the quotient $\Sigma/\sigma$ is biholomorphic to $\C P^1$, there are two cases to distinguish depending on whether the real involution $\rho \circ \sigma$ has fix points or not. In the first case the surface is of constant mean curvature in $\R^3,$ $S^3$ or $\mathcal H^3$ (with mean curvature $|H|>1$). If $\rho \circ \sigma$ has no fixed points then the corresponding immersion is of Babich-Bobenko type. We want to show the converse, i.e., that $\Sigma/\sigma \cong \C P^1$ for Babich-Bobenko tori. 
\end{Rem}

\subsection{The light cone model}

CMC surfaces in 3-dimensional space forms can also be described by associated families of flat $\SL(2,\C)$-connections $\nabla^\lambda,$ $\lambda \in \C_*$, on a rank $2$ bundle $\tilde V \rightarrow M$
\cite{HitchinHM,BabBob}. 
In the case of tori, these  families of flat connections can be described by (algebraic-geometric) spectral data
consisting of a (compact) hyper-elliptic curve $\tilde \Sigma $ (the spectral curve), two meromorphic differentials, and a holomorphic line bundle. 
In the case of Babich-Bobenko tori \cite{BabBob} $\tilde \Sigma$ is the spectral curve of a finite gap solution of the Cosh-Gordon equation and 
admits a real involution covering $\lambda\mapsto-\bar\lambda^{-1}$. Therefore $\tilde \Sigma$ hyperelliptic and of odd genus. 
In this alternate approach the light cone model as developed in \cite{BuPP, BuCa} is used.  Its relation to quaternionic holomorphic geometry can be found in \cite[$\S$5]{BuQu}, details of the computations  is also included in the thesis of Quintino \cite{Qui} and in \cite{Qui2}. We only recall the main constructions here. The Pl\"ucker estimate cannot be applied to this approach directly, since $\tilde \nabla^\lambda$ have singularities on $M$, corresponding to the intersection of the surface with the infinity boundary of $\mathcal H^3,$ see \cite{HH2}.

As in \cite{BuQu} we start with $\C^4$ equipped with a quaternionic structure, i.e., a complex anti-linear map
\[j\colon\C^4\to\C^4\]
with $j^2=-1 $ and identify $\C^4 \cong \H^2.$ Moreover, we choose a determinant $\text{det}\in\Lambda^4(\C^4)^*$ with 
\[j^*\text{det}=\overline{\text{det}}\]
and
\[\text{det}(e_1,e_2,je_1,je_2)=1\]
 for $\{e_1, e_2, e_3:= j e_1, e_4 := j e_2\}$ being the standard basis of $\C^4.$ The quaternionic structure induces 
a real structure on $\Lambda^2\C^4$ (also denoted by $j$ by abuse of notation) via
\[v\wedge w\mapsto j v\wedge jw,\]
and the determinant induces an inner product  $\langle.,.\rangle$ on $\Lambda^2\C^4$ by
\[\langle \alpha,\beta\rangle=\text{det}(\alpha\wedge \beta).\]
The $6$-dimensional real subspace $V$ is spannend by
\begin{equation}
\begin{split} v_1&= e_1\wedge e_3,\\ v_2&=  e_2\wedge e_4,\\ v_3&=  e_1\wedge e_2+e_3\wedge e_4, \\v_4&= ie_1\wedge e_2-ie_3\wedge e_4,\\v_5 &= e_1\wedge e_4+e_2\wedge e_3,\\ v_6&=  ie_1\wedge e_4-ie_2\wedge e_3.
\end{split}
\end{equation}
Restricted to $V$ the inner product $\langle.,.\rangle$ is of signature $(5,1)$.

For a general $(n+2)$-dimensional real vector space $V$ with inner product of signature $(n+1,1)$, the $n$-sphere can be naturally identified with projectivation 
$\P\mathcal L$ of the light cone 
\[\mathcal L=\{ v\in V\mid \langle v,v\rangle=0\}.\] Moreover, $\P \mathcal L $ is equipped with a natural conformal structure: For a lift $l$ of $\pi\colon \mathcal L\to \P\mathcal L$
the Riemannian metric $g_l$ is defined as
\begin{equation}\label{gsigma}g_l(X,Y):=\langle dl(X),dl(Y)\rangle.\end{equation}
The space of orientation preserving conformal transformations -- the Moebius group -- can be identified with
\[SO(n+1,1)^+:=\{g\in SO(n+1,1)\mid \langle g(v),v\rangle <0\text{ if }\langle v,v\rangle <0\}.\]

For $V$ being the real subspace of $\Lambda^2\C^4$,  a real non-zero lightlike vector of $V$ is given by a complex 2-plane in $\C^4$ (nullity) invariant under $j$ (reality), i.e., it gives rise to a quaternionic line in $(\C^4,j).$ This identifies the 4-sphere
with the quaternionic projective line $\H P^1$, and relates the quaternionic holomorphic geometry to the lightcone model, see \cite[$\S$4]{BuQu}. \\

Constant curvature subgeometries of the Moebius geometry $(\P\mathcal L, SO(5,1)^+)$ are specified by a choice $v_\infty\in V\setminus \{0\}.$ Such a choice provides a natural  lift $l$ of $\P \mathcal L$
 onto the subset \begin{equation}\label{subgeom}S_\infty:=\{[x]\in \P\mathcal L\mid \langle x, v_\infty\rangle =-1\},\end{equation}
and the induced Riemannian metric $g_l$ defined in  \eqref{gsigma} is of constant sectional curvature $-\langle v_\infty,v_\infty\rangle.$ The corresponding group of orientation preserving isometries of the subgeometry is then given by
\[SO(5,1)^+_\infty=\{g\in SO(5,1)^+\mid g(v_\infty)=v_\infty\},\]
and is  isomorphic to $SO(5)$ if $\langle v_\infty,v_\infty\rangle<0$ and isomorphic to $SO(4,1)$ if $\langle v_\infty,v_\infty\rangle>0.$\\

To define the associated family of connections, we need the mean curvature sphere congruence $S$ for the immersion $f: M \rightarrow S^4$. This  is a map $S$ from $M$ into the space of  oriented $2$-spheres in $S^4,$ such that at every $p \in M$ the corresponding $2$-sphere $S(p)$, touches the immersion at $f(p)$ 
and has the same oriented tangent plane, and the same mean curvature.
An oriented 2-sphere $S\subset \P\mathcal L$ is determined by an oriented (real) 4-dimensional vector space $V_S\subset V$ of signature $(3,1)$ via $S= \P V_S\cap\P\mathcal L.$
This space is uniquely determined by its orthogonal complement, $V_N:=V_S^\perp,$ which is a oriented real 2-plane with positive definite inner product, and therefore admits a unique compatible
complex structure \[J_N\colon V_N\to V_N,\quad J_N^2=-\text{id}.\]

A conformal immersion $f\colon M\to\P\mathcal L$ is naturally equipped with the real rank 4 subbundle
\[V_S\longrightarrow M\]
of the trivial rank 6 bundle $V$,
with complexification locally given by
\[V_S \otimes \C = \Span\{\hat f,\hat f_z,\hat f_{\bar z}, \hat f_{z,\bar z}\}\]
for some local lift $\hat f\colon U\subset M\to\mathcal L$ (and where $f_z=\tfrac{\partial f}{\partial z}$ etc.), see \cite{BuPP,BuCa,BuQu}.
The bundle $V_S$ has induced signature $(3,1)$ and a natural orientation. Therefore $V_S$ gives rise to a sphere congruence, i.e., to a smooth map into the space of oriented 2-spheres in $S^4.$
It can be computed (see \cite{BuPP, BuCa}) that the sphere congruence $V_S$ is the mean curvature sphere congruence, i.e., $(V_S)_p$ is the unique oriented 2-sphere in $S^4$ which touches (with orientation) the surface at $f(p)$ to second order.
Analogous to the classical case of surface geometry in euclidean 3-space $\R^3$, we consider the induced splitting of the trivial connection $d$ with respect to
\[V=V_S\oplus V_N,\] 
where $V_N :=V_S^\perp$ given by
\[d=\mathcal D+\mathcal N\]
 into diagonal part $\mathcal D$ and off-diagonal part $\mathcal N.$ 
While  $\mathcal D$ is a connection, $\mathcal N$ is tensorial.\\

Another related vector bundle $Z$ is the bundle of skew-symmetric maps of $(V,\langle.,.\rangle)$ which map $\R\hat f$ to $\Span\{\hat f,\hat f_z,\hat f_{\bar z}\}$ and vice versa vanishing on other components. 

With these notations we list a few further important 
properties of the mean curvature sphere congruence:
\begin{itemize}
\item (see \cite{BuPP}) $f$ is isothermic if and only if there exists $\eta\in\Omega^1(M,Z)$ with
\[d\eta=d^{\mathcal D}\eta+[\mathcal N\wedge \eta]=0;\]
\item (see \cite{BuCa}) the Willmore energy of $f$ is given by
\[\mathcal W(f)=-\tfrac{1}{4}\int_M \text{tr}( *\mathcal N\wedge \mathcal N)\]
where $*dz=idz, *d\bar z=-id\bar z;$
\item (see \cite{BPP,BuPP,BuCa}) a surface is Willmore if and only if
\[d^{\mathcal D}*\mathcal N=0,\]
and constrained Willmore if and only if there exists a $q\in\Omega^1(M, Z)$
satisfying $d^{\mathcal D}q=0$
and
\[d^{\mathcal D}*\mathcal N=2[q\wedge *\mathcal N];\]
$q$ is called the {\em Lagrange multiplier} of $f$;
\item (see \cite{BuPP} or \cite[$\S$3.3]{BuQu}) a surface $f$ has parallel mean curvature vector $H$ in the constant sectional curvature subgeometry $S_\infty$  of $\P\mathcal L$ defined by $v_\infty$ in \eqref{subgeom} if and only if
\[\mathcal D v_\infty^\perp=0,\]
where $v_\infty^\perp$ is the projection to $V_N=V_S^\perp;$ in particular, $f$ is minimal in $S_\infty$ if and only
if $v_\infty^\perp=0;$
\item a surface with parallel mean curvature vector $H$ in $S_\infty$ is constrained Willmore with Lagrange multiplier
$q$ which is determined by
\[q^{1,0}v_\infty^\perp:=-\mathcal N^{1,0} (v_\infty^\perp)^+,\] where $ v_\infty^\perp=(v_\infty^\perp)^++(v_\infty^\perp)^-$
is with respect to the decomposition of the normal bundle $V_N\otimes\C=V_N^+\oplus V_N^-,$
see \cite[$\S$7.2.2]{Qui}.
\end{itemize}
In particular, the Lagrange multiplier $q$ of a constrained Willmore surface $f$ is unique if and only if $f$ is non-isothermic, as 
for two Lagrange multipliers $q_1,q_2$ the 1-form 
$$\eta=*q_1-*q_2\in\Omega^1(M,Z)$$ solves $d\eta=0.$

The following theorem reduces the constrained Willmore property of a given immersion $f$ to the flatness of an associated family of flat connections in the language of the light cone model. 
\begin{Pro}\cite{BuPP,BuCa}\label{mcsfam}
The surface $f\colon M\to\P\mathcal L$ is constrained Willmore with Lagrange multiplier $q\in\Omega^1(M, Z)$
if and only if
\[d^\lambda_q:=\mathcal D+\lambda^{-1}\mathcal N^{(1,0)}+\lambda\mathcal N^{(0,1)}+(\lambda^{-2}-1) q^{(1,0)}+(\lambda^2-1) q^{(0,1)}\]
is flat for all $\lambda\in\C^*,$ where $(1,0)$ and $(0,1)$ are the complex linear and complex anti-linear parts of a 1-form.
\end{Pro}

\subsection{Compatibility of the quaternionic and the lightcone theory}\label{quatlc} 
The two approaches, the quaternionic and the lightcone one, towards the associated family of flat connections are in fact equivalent, as both associated families are gauge equivalent, when choosing suitable parameters. In order to provide a link between these families we need to relate the two different ways to obtain the mean curvature sphere congruence $S$.\\

Oriented $2$-spheres in quaternionic geometry are given by  complex structures $\tilde S$ of $V= M \times \H^2$. To be more precise a $2$-sphere is a map 
$$\tilde S \in \SL(4, \C) \quad \text{with}\quad \tilde S^2=-\text{Id}.$$
On the other hand, an oriented 2-sphere $S\subset \P\mathcal L$ is determined by an oriented 4-dimensional vector space $V_S\subset V$ of signature $(3,1)$ via $S =\P V_S\cap\P\mathcal L.$
Moreover, $V_N:=V_S^\perp$ is a oriented real 2-plane with positive definite inner product, and therefore admits a compatible
complex structure \[J_N\colon V_N\to V_N,\quad J_N^2=-\text{id}.\]
We therefore obtain a decomposition
\[\Lambda^2 \C ^4 = V\otimes\C=V_S\otimes \C\oplus V_N^+\oplus V_N^-\]
such that
\[V_N^\pm=\{v\in V_N\otimes\C\mid J_N v=\pm i v\}\]
are complex null lines that are complex conjugated to each other, i.e.,
\[jV_N^\pm=V_N^\mp.\]

In particular, $V_N^\pm$ gives rise to complex planes $W^\pm$ in $\C^4$ satisfying 
\[\C^4=W^+\oplus W^-.\]  
Hence, there exists a unique $\tilde S\in\SL(4,\C)$
with
\begin{equation}\label{quat2sphere}
\tilde S^2=-\text{id}\quad \tilde S_{\mid W^\pm}= \pm i\quad \tilde Sj=j\tilde S,\end{equation}
which is a $2$-sphere in $\H P^1$ in the quaternionic sense. 

Conversely, every quaternionic $2$-sphere $\tilde S$ determines its $\pm i$ eigenspaces 
$W_{\tilde S}^\pm$ which are interchanged via $j.$ They define complex null-lines $V^\pm_N$ satisfying $jV^\pm_N=V^\mp_N,$ and therefore
define a real oriented 2-plane of signature $(2,0).$ Its orthogonal complement in $V$ is a real oriented vector space $V_S$ of signature $(3,1)$, hence a 2-sphere in $\P\mathcal L.$

In the quaternionic description of $f \colon M \rightarrow S^4 \subset \H P^1$, we consider the quaternionic line bundle $L = f^*\mathcal T,$ the pull-back of the tautological bundle $\mathcal T$ of $\H P^1$,  which can be viewed as a $M$-family
of $j$-invariant complex 2-planes in $\C^4$ determined by $f\colon M\to\P\mathcal L.$ Its
  mean curvature sphere congruence
\[\tilde S_f\colon M\to \{\tilde S\in\SL(4,\C)\mid \tilde S \text{ satisfies } \eqref{quat2sphere}\}\]
is determined by the above  identifications.
We consider the bundle decomposition $$\C^4=W^+\oplus W^-$$ into the $\pm i$-eigenbundles of $\tilde S_f$,
and the decomposition of the trivial connection $d$ as
\[d=\mathcal D_{\tilde S}+\mathcal N_{\tilde S}\]
into ${\tilde S}_f$ commuting and anti-commuting parts, i.e., $\mathcal D_{\tilde S}{\tilde S}_f=0$ and $\mathcal N_{\tilde S}{\tilde S}_f=-{\tilde S}_f\mathcal N_{\tilde S}.$
Again $\mathcal D_{\tilde S}$ is a connection and $\mathcal N_{\tilde S}$ is tensorial. Moreover, $\mathcal D_{\tilde S}$ induces $\mathcal D$ on $\Lambda^2\C^4$, and $\mathcal N_{\tilde S}$ acts as $\mathcal N\in\Omega^1(M, \mathfrak{so}(\Lambda^2\C^4,\text{det}))$. Note that reality of $\mathcal N_{\tilde S}$ is equivalent to anti-commutation with $S_f$, For details, see \cite[$\S$ 4.5]{BuQu}. The Lagrange multiplier $q\in\Omega^1(M,Z)$ is then given by $\eta\in\Omega^1(M,\mathfrak{sl}(4,\C))$ satisfying
\[\text{image}(\eta)\subset L\subset \text{ker}( \eta) \quad d^{\mathcal D_{\tilde S}}\eta=0 \]
and the Euler-Lagrange equation of a CW surface with Lagrange multiplier $\eta$ is
\[d^{\mathcal D_{\tilde S}}*\mathcal N_{\tilde S}=2[\eta\wedge*\mathcal N_{\tilde S}].\]
Consequently, a surface is constrained Willmore in the 4-sphere if and only if
the connections 
\begin{equation}\label{dqS}d^\lambda_{\eta,{\tilde S}}:=\mathcal D_{\tilde S}+\lambda^{-1}\mathcal N_{\tilde S}^{1,0}+\lambda\mathcal N_{\tilde S}^{0,1}+(\lambda^{-2}-1) \eta^{1,0}+(\lambda^2-1) \eta^{0,1}\end{equation}
are flat for all $\lambda\in\C^*.$ Moreover, the induced family of flat connections on $\Lambda^2\C^4$ satisfies
\begin{equation}\label{la2d}
\Lambda^2 d^\lambda_{\eta,{\tilde S}}=d^\lambda_q\end{equation}
for $\eta$ corresponding to $q$ under the above identifications.

The following Proposition is proven in \cite[$\S$ 9]{Qui}, and is used below to determine the structure of the spectral curves of the Babich-Bobenko tori:
\begin{Pro}
The family of $\SL(4,\C)$-connections $\nabla^\mu$ as in \eqref{nabmu} is gauge equivalent to $d^\lambda_{\eta,{\tilde S}}$ for $\mu=\lambda^2.$
\end{Pro}

\subsection{The structure of the spectral curve of a Babich-Bobenko torus}\label{babostructure}
The aim of this section is
 to show
 the holonomy spectral curve of a Babich-Bobenko torus defined by the rank 4 family $\nabla^\mu$  has the same properties as the Cosh-Gordon spectral curve by taking the Lagrange multiplier $\eta$ corresponding to $q_\infty$ as defined in \cite[page 130]{Qui}. Note that we consider the immersions maps into $S^3 \subset S^4 \cong \H P^1$ (to make the relation to quaternionic holomorphic surface geometry transparent). 
We first study the structure of the spectral curve
 for the case of $H=0$, which is equivalent to the vanishing of the Langrange multiplier $\eta  \mathop{\widehat{=}}  q=0$. The case  $0<|H|<1$ is morally the same, though the details are slightly different, see  Section \ref{appendix}. The application of the Pl\"ucker estimate in Section \ref{sec:pluecker} below works totally analogous in both cases.

\begin{Pro}\label{pro:babo}
For a Babich-Bobenko torus $f\colon M\longrightarrow S^3$ with $H=0$, the associated constrained Willmore family of flat connections $\nabla^\mu$  is gauge equivalent
 to a  $\C_*$-family of flat  $\SL(4,\C)$-connections $\tilde\nabla^{\lambda}$ 
  of the form
\[\tilde\nabla^\lambda=d+\begin{pmatrix}\omega(\lambda)&0\\0&\overline{\omega(\bar\lambda^{-1})}\end{pmatrix}\]
with $\mu=\lambda^2$  through a $\lambda$-dependent family of complex gauge transformations, where 
\begin{equation}\label{lam11}\omega(\lambda)=\lambda^{-1}\omega_{-1}+\omega_0+\lambda\omega_1\in\Omega^1(M,\mathfrak{sl}(2,\C))\end{equation} with
$*\omega_{\pm1}=\pm i\omega_{\pm 1}.$

Moreover, 
\[d+\omega(-\lambda) \;\;\; \text{ and } \;\;\; d+\overline{\omega(\bar\lambda^{-1})}\]
are gauge equivalent for all $\lambda\in\C^*$, and the monodromies  of $d+\omega(-\lambda)$ along non-trivial elements of the first fundamental group have neither unimodular nor real eigenvalues  for generic $\lambda\in S^1.$
\end{Pro}
\begin{proof}
Let $f\colon M\to S^3\subset\P\mathcal L$ be a Babich-Bobenko torus with mean curvature $H = 0.$ Then the family  of connections
 \[\tilde\nabla^\lambda=d^\lambda_{0,S}\]
as defined in \eqref{dqS} with Lagrange multiplier $\eta=0$ is flat for all $\lambda \in \C_*.$

Because the Babich-Bobenko surface is minimal in the intersection with the hyperbolic space $S_\infty$, the parallel vector 
$v_\infty$ is space-like and is contained in the mean curvature sphere bundle $V$ for all $p\in M.$ 
In particular, we have $\mathcal N(v_{\infty})=0$. Hence, by Proposition \ref{mcsfam} and \eqref{la2d} together with $q=0$
 $v_\infty$ is parallel with respect to $\Lambda^2 \tilde\nabla^\lambda$ for all $\lambda\in\C^*$. 
Recall that the 3-sphere $S^3\subset S^4$ is determined by
a space-like vector $v$  
via
\[S^3=\{[x]\in\P\mathcal L\mid \langle x,v\rangle=0\}.\]
Thus, $v$ is contained in $V_N$ for all $p\in M$ and hence $v$ is also parallel with respect to $\Lambda^2 \tilde\nabla^\lambda$ for all $\lambda\in\C^*$. Note also that $v$ and $v_\infty$ are perpendicular.
There exists a conformal transformation of $S^4$ given by a real element of $\SO(\Lambda^2\C^4,\det)$ which transforms
the (real and space-like) 2-vectors $v$ and $v_\infty$ (which are perpendicular to each other as they define perpendicular 3-spheres in the 4-sphere) as follows.
\[v \mapsto  \tilde v:=e_1\wedge e_2+e_3\wedge e_4, \quad v_\infty\mapsto \tilde v_\infty:=ie_1\wedge e_2-ie_3\wedge e_4.\]
Hence, we can assume without loss of generality that $v=\tilde v$ and $v_\infty=\tilde v_\infty$ are parallel for all $\lambda\in\C^*.$
For a connection
$d+A$ with $A\in\Omega^1(M,\mathfrak{sl}(4,\C))$  the 2-vectors $v,w\in\Gamma(M,\Lambda^2\C^4)$ are parallel  if and only if it  $e_1\wedge e_2$ and $e_3\wedge e_4$ are parallel. This is equivalent to
$A$ being of the form
\[A=\begin{pmatrix}A_1&0\\0&A_2\end{pmatrix}\]
for $A_1,A_2\in\Omega^1(M,\mathfrak{sl}(2,\C))$ as can be seen as follows:
for $A=(a_{i,j})$
we get
\[ A (e_1\wedge e_2)= (A e_1)\wedge e_2+e_1\wedge (A e_2)=(a_{3,1} e_3+a_{4,1} e_4)\wedge e_2+e_1\wedge (a_{3,2} e_3+a_{4,2} e_4) \]
which vanishes if and only if \[a_{3,1}=a_{3,2}=a_{4,1}=a_{4,2}=0,\] and similarly for $A( e_3\wedge e_4).$
Moreover, if $A$ commutes with $j$ or equivalently $\Lambda^2(d+A)$ is real, then $A_2=\overline{A_1}:$
in fact, 
\[ j (d+A)(e_1)=jA_1(e_1)=j(a_{1,1} e_1+a_{2,1} e_2)=\overline{a_{1,1} }j e_1+\overline{a_{2,1}} j e_2=\overline{a_{1,1}} e_3+\overline{a_{2,1}} e_4\]
and \[(d+A)(je_1)= a_{3,3}e_3+a_{4,3}e_3,\]
and similarly for $e_2.$

Hence, with $q\widehat =\eta=0$ we see that $\tilde\nabla^\lambda$ has the form
\[\tilde\nabla^\lambda=d+\begin{pmatrix}\omega(\lambda)&0\\0&\overline{\omega(\bar\lambda^{-1})}\end{pmatrix},\]
where  $\omega(\lambda)$ is as stated in \eqref{lam11}. By
\cite[Lemma 9.14]{Qui} and \cite[Equation (2.11)]{Bohle}, $\tilde\nabla^\lambda$ is gauge equivalent to $\nabla^\mu$ as defined in \eqref{nabmu} (with $\mu=\lambda^2$). 

If the monodromies  of $d+\omega(-\lambda)$ along non-trivial elements $\gamma_i$ of $\pi_1(M)$ would have either unimodular or real eigenvalues for generic $\lambda\in S^1$
then \cite[Proposition 3.2]{Bohle} shows that the eigenvalues of $\nabla^\mu$ must be all equal to 1 for all $\mu \in \C_*$ and therefore this case can be excluded
  by \cite[Theorem 5.1]{Bohle}. 

It remains to prove that
$d+\omega(-\lambda)$ and $d+\overline{\omega(\bar\lambda^{-1})}$ are gauge equivalent for all $\lambda\in\C^*.$ 
We make use of the fact that
$\tilde\nabla^\lambda$ and $\tilde\nabla^{-\lambda}$ are gauge equivalent (as both are gauge equivalent to
$\nabla^{\mu = \lambda^2}$), and want to determine the gauge as explicit as possible.
Consider first the case that at some point $p\in M$
\[w=(e_1\wedge e_3-e_2\wedge e_4)\]
is (twice) the  oriented normal of $f$ at $p.$ Then, 
$S^+_p$ is the 2-plane determined by
\[w+ i \tilde v_\infty=(e_1-e_4)\wedge (e_3-e_2)\] and
$S^-_p$ is the 2-plane determined by
\[w-i \tilde v_\infty=(e_1+e_4)\wedge (e_2+e_3),\]
i.e.,
\[S^+_p=\text{span}(e_1-e_4,e_3-e_2)\]
and
\[S^-_p=\text{span}(e_1+e_4,e_2+e_3).\]
By \cite[$\S 3.2$ ]{BuQu} or \cite[Lemma 9.14]{Qui}, 
\[\tilde\nabla^{-\lambda}=\tilde\nabla^{\lambda}.H,\]
where $H\colon \C^4\to\C^4$ is determined  by
\[H(s_\pm)=\pm i s_\pm\quad\text{ for } \quad  s_\pm\in S^\pm.\]
Using the standard basis of $\C^4,$ $H_p$ is given by
\[H_p=\begin{pmatrix} 0&0&0&-i\\0&0&-i&0\\0&-i&0&0\\-i&0&0&0\end{pmatrix}.\]
Now, let (twice) the normal of $f$ at the point $q\in M$ be arbitrary, i.e. $N_q$ is in the real part of $\Lambda^2\C^4$
perpendicular to $\tilde v$ and $\tilde v_\infty$ and of length 2. There is a conformal transformation $\Psi_q$ of $S^4$
which fixes the 3-sphere and the sphere at infinity, and maps $N_p$ to $w.$ It must be (considered as a $\SL(4,\C)$-matrix commuting with $j$) of the form
\[\Psi_q=\begin{pmatrix} P_q &0\\0& \bar P_q\end{pmatrix}\]
where $P_q$ is a 2 by 2 matrix of unimodular determinant. Denote
\[K=\begin{pmatrix} 0&-i\\-i&0\end{pmatrix}.\]
Then,
\[H_q=\begin{pmatrix}P^{-1}_q&0\\0&\bar P^{-1}_q\end{pmatrix} \begin{pmatrix} 0 &K\\K& 0\end{pmatrix}\begin{pmatrix} P_q &0\\0& \bar P_q\end{pmatrix}= \begin{pmatrix} 0 &P^{-1}_qK\bar P_q\\ \bar P^{-1}_qKP_q& 0\end{pmatrix}.\]
Because the space of $\SL(4,\C)$ matrices commuting with $j$ and fixing 
\[\tilde v,\tilde v_\infty \quad \text{and} \quad e_1\wedge e_3-e_2\wedge e_4\]
is given by \[\begin{pmatrix}\alpha a&\alpha b&0&0\\\alpha \bar b&\alpha\bar a&0&0\\0&0&\alpha^{-1} \bar a&\alpha^{-1} \bar b\\0&0&\alpha^{-1}  b&\alpha^{-1} a\end{pmatrix},\]
where
where $a,b\in\C$, $\alpha\in S^1$ with
\[a\bar a-b\bar b=1,\]
 $P_q$ is unique up to
\[P_q\longmapsto  \alpha\begin{pmatrix}a&b\\\bar b&\bar a\end{pmatrix}.\]
Note that
\[\bar \alpha^{-1}\begin{pmatrix}\bar a&-\bar b\\-b&a\end{pmatrix}K\alpha\begin{pmatrix}a&b\\\bar b&\bar a\end{pmatrix}=\alpha^2K.\]
As $P$ can be locally chosen to be smooth on $M$
we find a well-defined global gauge transformation 
\[g\colon M\to\GL(2,\C)\] with unimodular determinant
which is locally given by
\[g_q=\bar P^{-1}_qKP_q\]
and satisfies
\[(d+\omega(\lambda)).g=d+\overline{\omega(\bar\lambda^{-1})}.\]
Moreover, due to the quadratic factor $\alpha^2,$ one can deduce that $g$ can actually be chosen to be $\SL(2,\C)$-valued.
\end{proof}

As an immediate corollary the spectral curve $\Sigma$ has the same properties as a Cosh-Gordon spectral curve. 
\begin{Cor}\label{coshgordon}
The spectral curve $\Sigma$ of a Babich-Bobenko torus (with $H=0$) considered as a Willmore torus, i.e, as the Riemann surface parametrizing the eigenlines of $\nabla^\mu,$ is given by a double covering of $ \lambda: \Sigma \longrightarrow \C P^1$ with $\lambda^2 =\mu$ satisfying:\begin{itemize}
\item $\lambda$ is branched over $\lambda = 0$ and $\lambda = \infty$
\item there exist two holomorphic functions -- the monodromy maps-- \[\nu_1,\nu_2\colon\Sigma\setminus\{0,\infty\}\longrightarrow\C^*\] 
such that the hyper-elliptic involution $\sigma$ satisfies 
\[\sigma^*\nu_i=\frac{1}{\nu_i},\quad \text{for }i = 1,2\]

\item $\Sigma$ has a anti-holomorphic involution $\rho$ covering $\lambda\mapsto-\bar\lambda^{-1}$ with \[ \rho^*\nu_i=\bar\nu_i, \quad \text{for }i = 1,2\]
\item $\mu : \Sigma \longrightarrow \C P^1$ is a  four-fold covering, i.e., for generic $ \mu \in \C_*$ the connection $\nabla^\mu$ has $4$ distinct eigenvalues along the generators $\gamma_k$ $k=1,2$ of $\pi_1(T^2)$ given by the four elements of the set
\[\{\nu_k(\xi) \mid  \quad \xi\in \Sigma : \mu=(\lambda(\xi))^2\}.\]
\end{itemize}
\end{Cor}
\begin{proof}
By the previous proposition, the spectral curve $\Sigma$ is given by the holonomy spectral curve of the family of flat connections $\hat\nabla^\lambda = d+\omega(\lambda).$ Since $\hat \nabla^\lambda$ is SL$(2, \C),$ the hyper-elliptic involution $\sigma$ maps an eigenvalue of the monodromy to its inverse. The other involution $\rho$  is induced by the quaternionic multiplication $\jj$, which covers $\lambda\mapsto-\bar\lambda^{-1}$ and (complex) conjugates the eigenvalues of the monodromy. Moreover, the parameter covering $\mu = \lambda^2$ is unbranched over $\C^*$. Thus the quotient $\Sigma/\sigma$ is biholomorphic to $\C P^1.$
\end{proof}

\subsection{Non-minimal Babich-Bobenko tori}\label{appendix}
We show a modified version of Corollary \ref{coshgordon} for Babich-Bobenko surfaces
\[f : M \longrightarrow S^3\]
with mean curvature $H \neq 0$ (and $|H| <1$) in the hyperbolic space $\mathcal H^3 \subset S^3$.
Again we use the notations as introduced in Section \ref{quatlc} (or \cite{Qui} for more details) and consider the $\C^*$-associated family of flat connections
\[d^{\lambda}_{\eta,S}\]
on the trivial $\C^4$-bundle for the Lagrange multiplier $\eta$ given by \[\eta=H\eta_\infty,\]
where $\eta_\infty$ is defined in \cite[Theorem 8.16]{Qui}. For further references see \cite{BuCa,BuQu,Qui2}.
The connections $d^{\lambda}_{\eta,S}$  induce the family of flat connections 
\[d^\lambda_{\eta}=\mathcal D+\lambda^{-1}\mathcal N^{(1,0)}+\lambda \mathcal N^{(0,1)}+(\lambda^{-2}-1)\eta^{(1,0)}+(\lambda^{2}-1)\eta^{(0,1)},\]
on the $\Lambda^2\C^4$
with Lagrange multiplier $\eta$.
By \cite[Lemma 9.14]{Qui} and \cite[Equation (2.11)]{Bohle} the connections $d^{\lambda}_{\eta,S}$ and the constrained Willmore associated family of flat connections 
$\nabla^\mu$ defined in \eqref{nabmu} are gauge equivalent for $\mu=\lambda^2$.

The surface $f$ is an isothermic constrained Willmore torus by assumption and admits a conserved quantity
\cite[Proposition 8.20]{Qui}. Since we are considering surfaces in $S^3\subset S^4$, there is for every $\lambda\in\C^*$ a complex 2-dimensional subspace of
$\Lambda^2\C^4$  on which
$d^\lambda_{\eta}$ acts trivially, see \eqref{la2d}. Applying a suitable $SL(4,\C)$-transformation (depending on $\lambda$ and $p\in M$)
we can assume without loss of generality that the invariant subspace is spanned by
\[v= e_1\wedge e_2,\ e_3\wedge e_4.\]
A short computation as in Section \ref{babostructure}  shows that $d^{\lambda}_{\eta,S}$ is of the form
\begin{equation}\label{blockform}d+\begin{pmatrix} A(\lambda) &0\\0& B(\lambda)\end{pmatrix}\end{equation}
for $A(\lambda),B(\lambda)\in \Omega^1(M,\mathfrak{sl}(2,\C))$ (with respect to the chosen $\lambda$-dependent frame). Recall that the reality condition implies that  $d^\lambda_{\eta}$ reduces to a $SO(3,1)$-connection for every $\lambda \in S^1$. This implies
\[\bar A(\lambda)=B(\lambda) \quad \text{ for } \lambda \in S^1.\]
\begin{Rem}
For CMC surfaces in $S^3$ the connection $d^\lambda_{\eta}$ reduces to a $SO(4)$-connection with reality condition $-\bar A^T=A$ and $-\bar B^T=B.$\end{Rem} 
\begin{Pro}
The spectral curve of a Babich-Bobenko torus $f: T^2 \longrightarrow S^3$
(corresponding to the Lagrange multiplier $\eta$ defined in \cite[Theorem 8.16]{Qui}) is a hyper-elliptic surface
\[\lambda\colon\Sigma\longrightarrow\CP^1\]
with an anti-holomorphic involution $\rho$ covering $\lambda\mapsto-\bar\lambda^{-1}.$
Moreover, $\Sigma$ is endowed with two meromorphic differentials $\theta_1,\theta_2$ of the second kind satisfying 
\[\sigma^*\theta_k=-\theta_k,\;\;\rho^*\theta_k=\bar\theta_k,\]
where $\sigma$ is the hyper-elliptic involution and two holomorphic functions $\nu_1,\nu_2\colon\Sigma\setminus\lambda^{-1}\{0,\infty\}$
with $d\log\nu_k=\theta_k$. 

The functions $\nu_i$ parametrizes the eigenvalues of $\nabla^\mu$ along the generator $\gamma_k,$ of the first fundamental group of $T^2.$  The (generically) four eigenvalues of $\nabla^\mu$ are given by 
\[\{\nu_k(\xi) \mid k= 1,2 \quad \text{ and } \quad \xi\in \Sigma : \mu=(\lambda(\xi))^2\}.\]
\end{Pro}
\begin{proof}
Let $\tilde \Sigma$ be the constrained Willmore spectral curve given as the parametrization of the (generically 4 distinct) eigenvalues of $\nabla^\mu.$ Since $\nabla^\mu$ is gauge equivalent to \eqref{blockform}, it is the direct sum of two flat SL$(2, \C)$-connections
\[d + A(\lambda) \quad \text{ and } d + B(\lambda)\]
for $\lambda^2=\mu.$
Let $h$ be an eigenvalue of the monodromy of $\nabla^\mu$. We assume without loss of generality that it is a eigenvalue of $d+A(\lambda)$. Since $d+A(\lambda)$ is a SL$(2,\C)$-connection, $h^{-1}$ is also a eigenvalue of $d+A(\lambda).$ Thus we can define an involution 
\[\sigma : \Sigma \longrightarrow \Sigma, \quad h \longmapsto h^{-1}\]
(note that $\sigma$ holomorphically extends to $\mu=0,\infty$ and therefore is well-defined on $\Sigma$).
Since the decomposition into blocks is valid for all $\mu\in\C^*$, the quotient $\Sigma/\sigma$ is $\CP^1.$ The remaining properties can be easily proved using  \cite[Proposition 3.1]{Bohle} together
with the reality condition $\bar A=B.$
\end{proof}
\begin{Rem}
Note that for $H\neq0$, the connections $d^{\lambda}_{\eta,S}$ and the connections given by \eqref{blockform}
are gauge equivalent by a  $\lambda$-dependend  gauge transformation. Thus, the map  $\lambda \colon\Sigma\longrightarrow\Sigma/\sigma$ is not necessarily branched over $0$ and $\infty$ as in the $H=0$ case. 
\end{Rem}

\section{Pl\"ucker estimates}\label{sec:pluecker}

We show that all isothermic constrained Willmore tori of Babich-Bobenko type have Willmore energy above $8\pi.$ The following Pl\"ucker estimate is relating the dimension of the holomorphic sections of $V/L$ (without monodromy) to the Willmore energy of the corresponding immersion. 

\begin{The}[\cite{FLPP}, Theorem 4.12]\label{Plücker}
Let $f:T^2 \longrightarrow S^3$ be a conformal immersion and $(V/L, D)$ be the quaternionic holomorphic line bundle associated to it. Let $k \in \N$ be the dimension of $H^0(T^2,V/L)$ (with trivial monodromy). Then a lower bound for the Willmore energy of $f$ is given by
$$W(f) \geq  \begin{cases}
    \pi k^2 & k \; even \\
     \pi (k^2 - 1) & k \; odd   .\end{cases}$$
     
\end{The}

\begin{Rem}
For every immersion into $S^3$ the sections $[1,0]$ and $[0,1]$ are holomorphic sections without monodromy. Thus $H^0(V/L)$ is at least $2$-dimensional. The most relevant cases in the following are: if there exist a third quaternionic linearly independent holomorphic section, then the Willmore energy of $f$ is at least $8 \pi$, if there exist a fourth quaternionic linearly independent holomorphic section, the lower bound is $16\pi.$
\end{Rem}

\begin{Rem}
We have shown in Section \ref{babostructure}  and Section \ref{appendix} that the associated family of flat connections $\nabla^\mu$ of a Babich-Bobenko torus has four distinct eigenvalues  for generic $\mu \in\C_*$. Thus by Bohle \cite{Bohle} $\hat L= $Ker$A_\circ$ is a non-constant quaternionic line subbundle of $V.$ 
\end{Rem}

\begin{Lem}\label{holoexistence}
Let $f: T^2 \longrightarrow S^3$ be a Babich-Bobenko torus with Willmore energy below $8\pi.$ Then
every branch point of the spectral curve $\lambda\colon\Sigma \longrightarrow \Sigma/\sigma \cong\C P^1$ except those over $\lambda=0$ and $\lambda=\infty$ corresponds to a non-constant holomorphic section of $V/L$ with $\Z_2$-monodromy.
\end{Lem}

\begin{proof}
The spectral curve $\Sigma$ is the surface parametrizing the eigenlines of $\nabla^\mu$ -- the constrained Willmore associated family of flat connections. It admits two involutions: $\sigma$ and $\rho.$ While the involution $\rho$ corresponds to the quaternionic multiplication by $\jj$ and is fixpoint free, the involution $\sigma$ maps a holomorphic section $\psi$ with monodromy $h$ to a holomorphic section with monodromy $h^{-1}$. Therefore the branch points of $\Sigma$ correspond to those $\nabla^\mu$-parallel sections $\psi$ of $V$ with $\Z_2$-monodromy, i.e., prolongations of holomorphic sections of $V/L$ with $\Z_2$-monodromy. It is thus crucial to show that these $\psi$ are non-constant sections of $V$, which clearly holds whenever the monodromy of the section $\psi$ is non-trivial. Thus we restrict to the case where $\psi$ has trivial monodromy.

Let $\lambda_0$ be a branch point of $\Sigma$ and $\psi_0$ be the $\nabla^{\mu_0}$-parallel section of $V$ without monodromy associated to $\lambda_0$ where $\mu_0=\lambda_0^2$.  If $ \mu_0 \neq 1$ then $\psi$ must be non-constant, since $\psi$ being constant would imply that $\tilde \psi \in \Gamma($Ker${A_\circ})$ and then $\hat L= $Ker$A_\circ$ would be constant in contradiction to \cite[Theorem 5.3]{Bohle}.

It remains to show that also for the case  $\mu_0=\lambda_0^2=1$
there exists
a non-constant section $\tilde \psi$, given by a prolongation of a holomorphic section of $V/L,$ with trivial monodromy.
Let $\nabla^\mu$ be the CW associated family, and $\mu_0=1$ be a branch point of the spectral curve $\Sigma$. 
Because of the $\rho$-symmetry (interchanging  the two points over $\mu_0=1$) $\Sigma$ is not totally branched at $\mu_0$, i.e., we can use a local coordinate
$\xi$ on $\Sigma$ with $\xi^2=\mu-1.$
Assume that the Willmore energy is below $8\pi.$
By \cite[Theorem 4.3 (iii)]{BLPP}  there is 
a smooth family of $\nabla^{\mu(\xi)}$-parallel sections $\psi^\xi$ parametrized on an open subset of $\Sigma$ around $\xi=0$ depending smoothly on $\xi$, i.e., we have
\[\nabla^{\mu(\xi)}\psi^\xi=0.\]

Differentiating this equation with respect to $\xi$ (denoted by $()'$)  at $\xi=0$ (and therefore  $\mu=1$) gives 

\[0=(\nabla^{\mu(x)})'\psi+d \psi'=\big(\frac{d \mu(\xi)}{d\xi}_{\mid\xi=0} A_\circ^{1,0}+\frac{d \mu^{-1}(\xi)}{d\xi}_{\mid\xi=0}  A_\circ^{0,1}\big)\psi+d\psi'=d\psi',\]

where $\psi=\psi^{\xi=0}$ is constant in $z \in T^2$(as $\nabla^{\mu=1}=d$). Differentiating once more and evaluating at $\xi=0$ thus gives

\[0=(\nabla^{\mu(x)})''\psi+2(\nabla^{\mu(x)})'\psi'+d \psi''=(-2A_\circ^{1,0}+2A_\circ^{0,1})\psi+d \psi''.\]

Since $(-2A_\circ^{1,0}+2A_\circ^{0,1})\psi$ is contained in $L = f^* \mathcal T =$ Im $(A_\circ)$, $d\psi''$ is contained in $L$ as well showing that $\psi''$ is the prolongation of a (locally) holomorphic section of $V/L$.
Since the monodromy takes the value 1 with at least second order (because the monodromy is trivial at the branch point $\mu = 1$),  $\psi''$ has also trivial monodromy. If $\psi''$ would be constant then
\[0=(-2A_\circ^{1,0}+2A_\circ^{0,1})\psi,\]
which yields that Ker$A_\circ$ is constant giving a contradiction. 
\end{proof}

\begin{Lem}\label{linearindependence}
Two holomorphic sections of $V/L$ with non-trivial $\Z_2$-monodromy corresponding to different branch points of $\lambda\colon\Sigma\to\Sigma/\sigma=\CP^1$ not lying over $0$ or $\infty$, which are not interchanged by the involution $\rho,$ are quaternionic linear independent.
\end{Lem}

\begin{proof}
Let $\tilde \psi_1$ and  $\tilde\psi_2$ be two 
holomorphic sections of $V/L$ with $\Z_2$-monodromy. If these sections have different $\Z_2$-monodromies, then they are clearly quaternionic linear independent. Thus let the $\psi_i$ have the same non-trivial $\Z_2$-monodromy in the following.

Due to \cite[Section 2.5]{Bohle} and the fact that their monodromy is non-trivial on $T^2$, it is enough to prove that their prolongations $\psi_1$ and $ \psi_2$, which are parallel sections with respect to
$\nabla^{\mu_1}$ and $\nabla^{\mu_2}$ (corresponding to the branch points $\lambda_1$ and $\lambda_2$ of $\Sigma\longrightarrow\Sigma/\sigma$), are linear independent. If $\mu_1=\mu_2$ it follows from Proposition \ref{pro:babo} that $\psi_1$ and $\psi_2$ are (quaternionic) linear independent since  $\xi_1\neq\xi_2$ and $\xi_1\neq\rho(\xi_2).$ 

If $\mu_1\neq\mu_2$ we obtain that $\tilde \psi_1$ and $\tilde\psi_2$ are complex linear independent. Assume that they are not independent as quaternionic sections, then we would have w.l.o.g.
 
 \[\tilde \psi_1=a\tilde \psi_2+b\tilde \psi_2\jj\]
 for some $a,b \in \C.$ Moreover, from 
 \[\nabla^{\mu_1}\tilde \psi_1=0; \;\;\; \nabla^{\mu_2}\tilde\psi_2=0\;\;\; \nabla^{\bar\mu_2^{-1}}\tilde \psi_2\jj=0\]
 we obtain 
\begin{equation}
\begin{split}
0= &(\mu_1^{-1}-\mu_2^{-1}) A_\circ^{1,0} a\tilde\psi_2\\
 +&(\mu_1^{-1}-\bar\mu_2) A_\circ^{1,0} a\tilde\psi_2\jj\\
 +&(\mu_1-\mu_2) A_\circ^{0,1} a\tilde\psi_2\\
 +&(\mu_1-\bar\mu_2^{-1}) A_\circ^{1,0} a\tilde\psi_2\jj.\\
\end{split}
\end{equation}
By type decomposition (see \cite[Section 2.1]{FLPP}) we obtain $A_\circ\psi_2=0$. Since $\psi_2$ is $\nabla^{\mu_2}$-parallel this implies $\psi_2$ being constant, which is a contradiction by \cite[Theorem 5.3]{Bohle}
\end{proof}

\begin{Lem}\label{spinstructure}
Let $f\colon T^2\longrightarrow S^3$ be a constrained Willmore torus of Babich-Bobenko type with spectral genus $g \geq 3.$
Then, either one of the branch points of the spectral curve over the unit disc $D \subset \C$ corresponds to the trivial monodromy or at least two of the branch points of the spectral curve on the punctured unit disc $D_*$ correspond to the same (non-trivial) $\Z_2$-monodromy.
\end{Lem}

\begin{proof}
For $g >3$ we have at least 5 branch points over the punctured unit disc 
$$D:=\{\lambda\in \C\mid 0<|\lambda|\leq1\}.$$
The claim follows from the fact that there exist only $3$ different non trivial spin structures of the torus.

It remains to show the Lemma in the case of $g=3$, where we have $4$ branch points over the unit disc (that are not interchanged by $\rho$). Assume that none of the $4$ branch points on the unit disc corresponds to the trivial spin structure and moreover, for $\mu \neq 0 $ the other $3$ branch points correspond to different spin structures. Then the spin structure at $\mu = 0$ must coincide with the one at $P_k$ for a $k \in \{1, 2, 3\}$, since there exist only $3$ different non-trivial spin structures of a torus. Without loss of generality we can assume $k=1$. We want to show that the spin structures corresponding to $P_2$ and $P_3$ must then coincide.

In this case, the closed non-trivial curve, the green curve in Figure \ref{fig:g3spec}, through the branch points $P_2$ and $P_3$ denoted by $\gamma_{23}$,  is homologous to the difference of the closed (red) curve $\gamma_S$ through the Sym-points $S_1$ and $S_2$ and the closed (blue) curve $\gamma_{01}$ connecting $0$ and $P_1$. Let  $\theta_i = d\log \nu_i$ ($i=1,2$)  be the logarithmic differentials of the monodromy maps $\nu_i$ and consider integrals of $\theta_i$ along these curves. Using the hyper-elliptic symmetry we want to show that

\[2 \int_{P_2}^{P_3} \theta_i = \int_{\gamma_{23}} \theta_i \in 4 \pi i \Z\]

Since $0$ and $P_1$ correspond to the same spin structure by assumption we first show that
 \[
 \int_{\gamma_{01}} \theta_k \in 4 \pi i \Z\]
 for $k=1,2$. As $\theta_k$ has trivial residue, we
 can interpret the above integral as the integral of $\theta_k$ along any curve homotopic to $\gamma_{01}$ which does not pass through $0\in\Sigma$.
 
 In order to analyze the integral, we apply a renormalization:
 For $k=1,2$ there exist a closed 1-form $\eta_k$ on $\Sigma\setminus\{0\}$ with support in a small neighborhood of $0$ which satisfies $\sigma^*\eta_k=-\eta_k$ and such that \[\theta_k+\eta_k\]
 extends smoothly through $0,$ compare with the limiting analysis of \cite[Proposition 3.10]{HitchinHM}. Note that $\int_{\gamma_{01}}\eta_k=0$. Using an analogous computation as in \cite[Proposition 3.10]{HitchinHM} again, we can 
 associate to $0\in\Sigma$  renormalized eigenvalues $\tilde \nu_1$ and $\tilde \nu_2$  which take values in $\{\pm1\}$
 and encode the spin structure of the surface. The sign is encoded in the parity of the constant part of the expansion
 in \cite[Proposition 3.10]{HitchinHM}. Then, we obtain
 \[ \int_{\gamma_{01}} \theta_k=\int_{\gamma_{01}} \theta_k+\eta_k=2\int_{\gamma_{01}^+} \theta_k+\eta_k\in 4\pi i\Z,\]
 where $\gamma_{01}^+$ is given by a part of $\gamma_{01}$ which goes from $0$ to $P_1$, and the last equality follows
 from the fact that the values $ \nu_k(0)$ and $\nu_k(P_1)$ coincide as $0$ and $P_1$ correspond to the same spin structure.

Thus it remains to prove that the integral  of $\theta_i$ along the red curve $\gamma_S$ satisfies
 $$\int_{\gamma_{S}} \theta_k \in 4\pi i\Z.$$ 
 This follows from  the $\rho$-symmetry of the spectral data:  the integral of $\theta_k$ along the red curve $\gamma_S$ is  twice the integral along the curve $\tilde \gamma$ which is defined to be the part of $\gamma_S$ from the  point $S_1$ lying over the Sym point $\mu_1$ to the point $S_2$ lying over $-\bar \mu_1^{-1}$, i.e.,
 \[\int_{\gamma_S}\theta_k=\int_{\tilde\gamma}\theta_k+\int_{\sigma\circ\rho(\tilde\gamma)}\theta_k=\int_{\tilde\gamma}\theta_k-\int_{\tilde\gamma}\overline{\theta_k}=2\int_{\tilde\gamma}\theta_k,\]
 where the last equality uses the fact that the integral takes imaginary values. The well-definedness of $f$ then gives $\int_{\tilde\gamma}\theta_k\in2\pi i\Z$ proving the claim.
\end{proof}

\begin{The}\label{H3}
The Willmore energy of a constrained Willmore torus $f: T^2 \longrightarrow S^3$ of Babich-Bobenko type is at least $8\pi.$
\end{The}
\begin{figure}
\vspace{-0cm}
\hspace{2.5cm}
\includegraphics[width=1.\textwidth]{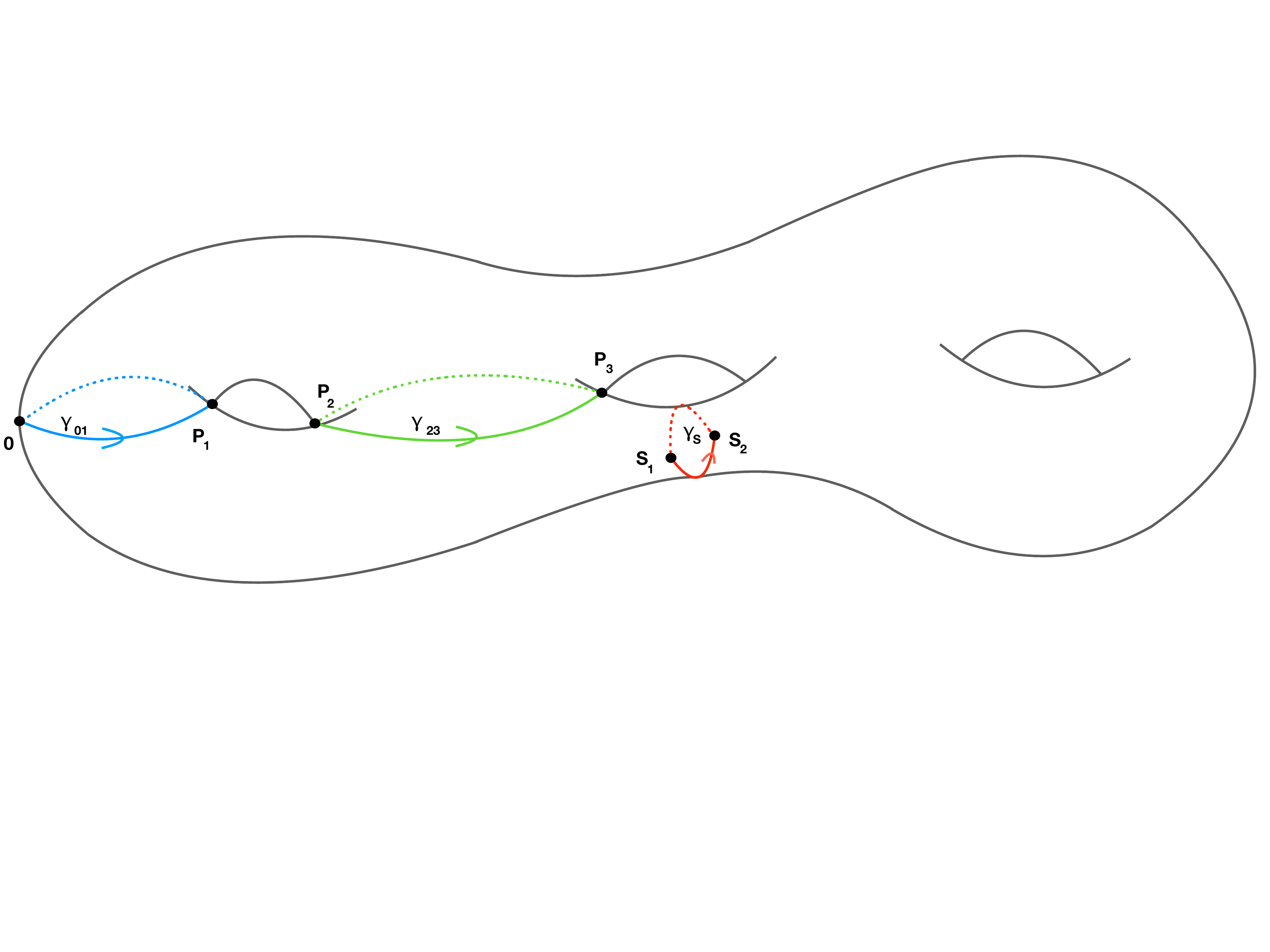}
\vspace{-3.5cm}
\caption{
\footnotesize
A spectral curve of genus 3, with branch points and certain cycles.}
\label{fig:g3spec}
\end{figure}

\begin{proof}
Since the spectral curve $\Sigma$ admits an involution $\rho$ covering $\lambda\mapsto-\bar\lambda^{-1}$, it must be of odd genus $g$, or $\lambda$ is unbranched. For $g\in \{0,1\}$ the surface $f$ is equivariant, see \cite{Hel1}. For $g=0$ the surface must be homogenous. This case cannot appear, since it would be a surface entirely contained in hyperbolic 3-space. For $g=1$ the surface is rotational symmetric and obtained by rotating a closed wavelike elastic curve in the hyperbolic plane $\mathcal H^2$ around the infinity boundary of $\mathcal H^2.$ The only periodic solution in this class is the family of elastic figure-$8$ curves in $\mathcal H^2$. These surfaces are non-embedded, see for example \cite{Hel2} or \cite{Steinberg}, and therefore they have Willmore energy above $8\pi$ by \cite{LiYau}.

Let $g \geq 3$. By Lemma \ref{holoexistence} we can associate to every branch point of $\Sigma$ a holomorphic section $\psi$ with $\Z_2$-monodromy of $V/L$. There exist exactly $4$ possible $\Z_2$-monodromies for $\psi$ arising from the $4$ different spin structures of $T^2$. To be more concrete, the two monodromy maps $\nu_i$ of $\psi$ satisfies:
$$(\nu_1(\psi),\nu_2(\psi))\in\{(1,1),(1,-1),(-1,1),(-1,-1)\}.$$ 

Every $\psi$ with $\pm 1$ monodromy gives rise to a proper holomorphic section of $V/L$ considered as a bundle over a suitable double cover $\tilde T^2$ of $T^2.$ Thus the theorem follows from the previous Lemma by applying the Pl\"ucker estimate (Theorem \ref{Plücker}). If the trivial monodromy arises over $\lambda =0$, the immersion $f$ has trivial spin structure and by \cite{PinkallHomotopy}  the surface cannot be embedded and hence its Willmore energy is at least $8\pi$. 
\end{proof}



\bibliographystyle{amsplain}
\bibliography{references}

\providecommand{\bysame}{\leavevmode\hbox to3em{\hrulefill}\thinspace}
\providecommand{\MR}{\relax\ifhmode\unskip\space\fi MR }
\providecommand{\MRhref}[2]{%
  \href{http://www.ams.org/mathscinet-getitem?mr=#1}{#2}
}
\providecommand{\href}[2]{#2}
\begin{thebibliography}{10}

\bibitem{Alexandrov}
A.D. Alexandrov, \emph{Uniqueness theorems for surfaces in the large v}, Amer.
  Math. Soc. Transl. \textbf{21} (1962), 412--416.

\bibitem{AndrewsLi}
Ben Andrews and Haizhong Li, \emph{Embedded constant mean curvature tori in the
  three-sphere}, J. Differential Geom. \textbf{99} (2015), no.~2, 169--189.
  \MR{3302037}

\bibitem{BabBob}
M.~Babich and A.~Bobenko, \emph{Willmore tori with umbilic lines and minimal
  surfaces in hyperbolic space}, Duke Math. J. \textbf{72} (1993), no.~1,
  151--185. \MR{1242883}

\bibitem{Bohle}
Christoph Bohle, \emph{Constrained {W}illmore tori in the 4-sphere}, J.
  Differential Geom. \textbf{86} (2010), no.~1, 71--131. \MR{2772546}

\bibitem{BLPP}
Christoph Bohle, Katrin Leschke, Franz Pedit, and Ulrich Pinkall,
  \emph{Conformal maps from a 2-torus to the 4-sphere}, J. Reine Angew. Math.
  \textbf{671} (2012), 1--30. \MR{2983195}

\bibitem{BohPet}
Christoph Bohle and G.~Paul Peters, \emph{Soliton spheres}, Trans. Amer. Math.
  Soc. \textbf{363} (2011), no.~10, 5419--5463. \MR{2813421}

\bibitem{BPP}
Christoph Bohle, G.~Paul Peters, and Ulrich Pinkall, \emph{Constrained
  {W}illmore surfaces}, Calc. Var. Partial Differential Equations \textbf{32}
  (2008), no.~2, 263--277. \MR{2389993}

\bibitem{Brendle}
Simon Brendle, \emph{Embedded minimal tori in {$S^3$} and the {L}awson
  conjecture}, Acta Math. \textbf{211} (2013), no.~2, 177--190. \MR{3143888}

\bibitem{BFLPP}
F.~E. Burstall, D.~Ferus, K.~Leschke, F.~Pedit, and U.~Pinkall, \emph{Conformal
  geometry of surfaces in {${\it S}^4$} and quaternions}, Lecture Notes in
  Mathematics, vol. 1772, Springer-Verlag, Berlin, 2002. \MR{1887131}

\bibitem{BuPP}
Francis Burstall, Franz Pedit, and Ulrich Pinkall, \emph{Schwarzian derivatives
  and flows of surfaces}, Differential geometry and integrable systems
  ({T}okyo, 2000), Contemp. Math., vol. 308, Amer. Math. Soc., Providence, RI,
  2002, pp.~39--61. \MR{1955628}

\bibitem{BuCa}
Francis~E. Burstall and David M.~J. Calderbank, \emph{Conformal submanifold
  geometry i-iii}, 2010.

\bibitem{BuQu}
Francis~E. Burstall and \'{A}urea~C. Quintino, \emph{Dressing transformations
  of constrained {W}illmore surfaces}, Comm. Anal. Geom. \textbf{22} (2014),
  no.~3, 469--518. \MR{3228303}

\bibitem{FLPP}
D.~Ferus, K.~Leschke, F.~Pedit, and U.~Pinkall, \emph{Quaternionic holomorphic
  geometry: {P}l\"{u}cker formula, {D}irac eigenvalue estimates and energy
  estimates of harmonic {$2$}-tori}, Invent. Math. \textbf{146} (2001), no.~3,
  507--593. \MR{1869849}

\bibitem{Hel2}
Lynn Heller, \emph{Constrained {W}illmore tori and elastic curves in
  2-dimensional space forms}, Comm. Anal. Geom. \textbf{22} (2014), no.~2,
  343--369. \MR{3210758}

\bibitem{Hel1}
\bysame, \emph{Equivariant constrained {W}illmore tori in the 3-sphere}, Math.
  Z. \textbf{278} (2014), no.~3-4, 955--977. \MR{3278899}

\bibitem{Hel3}
\bysame, \emph{Constrained {W}illmore and {CMC} tori in the 3-sphere},
  Differential Geom. Appl. \textbf{40} (2015), 232--242. \MR{3333105}

\bibitem{HH2}
Lynn Heller and Sebastian Heller, \emph{Higher solutions of hitchin's
  self-duality equations}, 2018.

\bibitem{HitchinHM}
N.~J. Hitchin, \emph{Harmonic maps from a {$2$}-torus to the {$3$}-sphere}, J.
  Differential Geom. \textbf{31} (1990), no.~3, 627--710. \MR{1053342}

\bibitem{KilianSchmidtSchmitt1}
Martin Kilian, Martin~U. Schmidt, and Nicholas Schmitt, \emph{Flows of constant
  mean curvature tori in the 3-sphere: the equivariant case}, J. Reine Angew.
  Math. \textbf{707} (2015), 45--86. \MR{3403453}

\bibitem{KoKuSo}
Nicholas~J. Korevaar, Rob Kusner, and Bruce Solomon, \emph{The structure of
  complete embedded surfaces with constant mean curvature}, J. Differential
  Geom. \textbf{30} (1989), no.~2, 465--503. \MR{1010168}

\bibitem{LiYau}
Peter Li and Shing~Tung Yau, \emph{A new conformal invariant and its
  applications to the {W}illmore conjecture and the first eigenvalue of compact
  surfaces}, Invent. Math. \textbf{69} (1982), no.~2, 269--291. \MR{674407}

\bibitem{PinkallHomotopy}
U.~Pinkall, \emph{Regular homotopy classes of immersed surfaces}, Topology
  \textbf{24} (1985), no.~4, 421--434. \MR{816523}

\bibitem{Pirola}
Gian~Pietro Pirola, \emph{Monodromy of constant mean curvature surface in
  hyperbolic space}, Asian J. Math. \textbf{11} (2007), no.~4, 651--669.
  \MR{2402943}

\bibitem{Qui}
A.~Quintino, \emph{Constrained willmore surfaces}, Ph.D. thesis, University of
  Bath, 2009.

\bibitem{Qui2}
\'{A}urea Quintino, \emph{Spectral deformation and {B}\"{a}cklund
  transformation of constrained {W}illmore surfaces}, Differential Geom. Appl.
  \textbf{29} (2011), no.~suppl. 1, S261--S270. \MR{2832028}

\bibitem{Richter}
J.~Richter, \emph{Conformal maps of a riemannian surface into the space of
  quaternions}, Ph.D. thesis, TU Berlin, 1997.

\bibitem{Schaetzle2}
Reiner~M. Sch{\"a}tzle, \emph{Conformally constrained willmore immersions},
  Advances in Calculus of Variations \textbf{6} (2013), 375--390.

\bibitem{Steinberg}
D.H. Steinberg, \emph{Elastic curves in hyperbolic space}, Ph.D. thesis, Case
  Western Reserve University, 1995.

\end{thebibliography}

\end{document}